\documentclass[12pt]{article}
\usepackage{}

\usepackage{epsfig}
\usepackage{latexsym}
\usepackage{caption}
\usepackage{amsfonts}
\usepackage{amssymb}
\usepackage{mathrsfs}
\usepackage{amsmath}
\usepackage{enumerate}
\usepackage{graphics}
\usepackage{MnSymbol}
\usepackage{float}
\usepackage{pict2e}
\usepackage{subfigure}

\usepackage{color}

\makeatletter

\newcommand{\Rmnum}[1]{\expandafter\@slowromancap\romannumeral #1@}

\makeatother

\begin{document}

\newtheorem{theorem}{Theorem}
\newtheorem{observation}[theorem]{Observation}
\newtheorem{corollary}[theorem]{Corollary}
\newtheorem{algorithm}[theorem]{Algorithm}
\newtheorem{definition}[theorem]{Definition}
\newtheorem{guess}[theorem]{Conjecture}
\newtheorem{claim}[theorem]{Claim}
\newtheorem{problem}[theorem]{Problem}
\newtheorem{question}[theorem]{Question}
\newtheorem{lemma}[theorem]{Lemma}
\newtheorem{proposition}[theorem]{Proposition}
\newtheorem{fact}[theorem]{Fact}

\makeatletter
  \newcommand\figcaption{\def\@captype{figure}\caption}
  \newcommand\tabcaption{\def\@captype{table}\caption}
\makeatother

\newtheorem{acknowledgement}[theorem]{Acknowledgement}

\newtheorem{axiom}[theorem]{Axiom}
\newtheorem{case}[theorem]{Case}
\newtheorem{conclusion}[theorem]{Conclusion}
\newtheorem{condition}[theorem]{Condition}
\newtheorem{conjecture}[theorem]{Conjecture}
\newtheorem{criterion}[theorem]{Criterion}
\newtheorem{example}[theorem]{Example}
\newtheorem{exercise}[theorem]{Exercise}
\newtheorem{notation}{Notation}
\newtheorem{solution}[theorem]{Solution}
\newtheorem{summary}[theorem]{Summary}

\newenvironment{proof}{\noindent {\bf
Proof.}}{\rule{3mm}{3mm}\par\medskip}
\newcommand{\remark}{\medskip\par\noindent {\bf Remark.~~}}
\newcommand{\pp}{{\it p.}}
\newcommand{\de}{\em}
\newcommand{\mad}{\rm mad}
\newcommand{\qf}{Q({\cal F},s)}
\newcommand{\qff}{Q({\cal F}',s)}
\newcommand{\qfff}{Q({\cal F}'',s)}
\newcommand{\f}{{\cal F}}
\newcommand{\ff}{{\cal F}'}
\newcommand{\fff}{{\cal F}''}
\newcommand{\fs}{{\cal F},s}
\newcommand{\s}{\mathcal{S}}
\newcommand{\G}{\Gamma}
\newcommand{\g}{(G_3, L_{f_3})}
\newcommand{\wrt}{with respect to }
\newcommand {\nk}{ Nim$_{\rm{k}} $  }

\newcommand{\q}{\uppercase\expandafter{\romannumeral1}}
\newcommand{\qq}{\uppercase\expandafter{\romannumeral2}}
\newcommand{\qqq}{\uppercase\expandafter{\romannumeral3}}
\newcommand{\qqqq}{\uppercase\expandafter{\romannumeral4}}
\newcommand{\qqqqq}{\uppercase\expandafter{\romannumeral5}}
\newcommand{\qqqqqq}{\uppercase\expandafter{\romannumeral6}}

\newcommand{\qed}{\hfill\rule{0.5em}{0.809em}}

\newcommand{\var}{\vartriangle}

\title{{\large \bf The Alon-Tarsi number of   planar graphs  }}
%\title{{The properties of the hypercube-like}}

\author{  Xuding Zhu\thanks{Department of Mathematics, Zhejiang Normal University,  China.  E-mail: xudingzhu@gmail.com. Grant Number: NSFC 11571319.}}

\maketitle

\begin{abstract}
	
	This paper proves that the Alon-Tarsi number of any planar graph is at most $5$, which gives an alternate  proof  of the $5$-choosability as well as the $5$-paintability of   planar graphs.


\end{abstract}

%% \linenumbers

%% main text
%\section{Result}

\section{Introduction}

Assume $G$ is a graph.  We associate to each vertex $v$ of $G$  a variable $x_v$. The graph polynomial $P_G(\vec{x})$ of $G$ is defined as    
$$P_G(\vec{x}) = \prod_{ u \sim v, u < v}(x_v-x_u),$$ 
where $\vec{x} = \{ x_v: v \in V(G)\}$ and $``<"$ is an arbitrary  fixed ordering of the vertices of $G$.
It is easy to see that a mapping $\phi: V \to R$ is a proper colouring of $G$ if and only if 
$P_G(\phi) \ne 0$, where $P_G(\phi)$ means to evaluate the polynomial at $x_v=\phi(v )$ for $v \in V(G)$.  
Thus to find a proper colouring of $G$ is equivalent to find an assignment of $\vec{x}$ so that the polynomial evaluated at this assignment is non-zero. The Combinatorial Nullstellensatz gives a sufficient condition for the existence of such an assignment.

Assume $P(\vec{x})$  a polynomial   with variable set $X$.  An  {\em   index function }  $\eta$ for $P({\vec{x}})$ is a mapping which assigns to each variable $x$ a non-negative integer $\eta(x)$. Given an index function $\eta$, we denote by ${\vec{x}}^{\eta}$ the monomial 
 $\prod_{x \in X}  x^{\eta(x )}$, and denote by $c_{P, \eta}$  the coefficient of ${\vec{x}}^{\eta}$  in the expansion of $P(\vec{x})$. The Combinatorial Nullstellensatz asserts that
 if $\eta$ is an index function with $\sum_{x \in X}\eta(x) $ equals the degree of $P(\vec{x})$ and $c_{P, \eta} \ne 0$,  and  $A_x$ is a set of $\eta(x)+1$ real numbers ( or elements of a field) for each $x \in X$, then   there is an assignment $\phi$ such that $\phi(x ) \in A_x$ for each $x \in X$ and $P(\phi) \ne 0$. In particular, if $c_{P_G, \eta} \ne 0$ and $\eta(x_v) < k$ for all $v \in V$, then $G$ is $k$-choosable.
 This method developed by Alon and Tarsi is now a powerful tool in the study of list colouring of graphs. 
Jensen and Toft  \cite{JT} defined the {\em Alon-Tarsi number}  of $G$ as 
 $$AT(G) = \min \{k: \text{$c_{P_G, \eta} \ne 0$ for some   index function $\eta$  with $\eta(x_v) < k$ for all $v \in V(G)$}\}.$$

As  observed in   \cite{HE}, $AT(G)$ has some distinct features and it is of interest to study $AT(G)$ as a separate graph invariant.
 Let $ch(G)$ be the choice number of $G$ and $\chi_P(G)$ be  the paint number (or the online choice number) of $G$ (cf. \cite{Scha} and \cite{Zhu}).  It follows from a result of Alon- and Tarsi   \cite{AT} and a generalization  of this result by Schauz  \cite{Sch}    $ch(G) \le \chi_P(G) \le  AT(G)$ for any graph $G$.  
   There are graphs for which both inequalities   are strict. However,   upper bounds for the choice number of many natural classes of graphs are also upper bounds for their Alon-Tarsi number. Thomassen \cite{TH} proved that every planar graph $G$ has $ch(G) \le 5$, and Schauz \cite{Scha} showed that  every planar graph $G$ has $\chi_P(G) \le 5$. A natural question (cf.  \cite{HE}) is whether $AT(G) \le 5$ for every planar graph $G$. In this note we answer this question in affirmative.

\begin{theorem}
	\label{thm-main0} If $G$ is a planar graph, then $AT(G) \le 5$. 
\end{theorem}

\section{A proof of Theorem \ref{thm-main0}}

The proof of Theorem \ref{thm-main0} is parallel to Thomassen's proof   of the $5$-choosability of planar graphs in \cite{TH}.

For simplicity, we write $c_{G,\eta}$ for $c_{P_G, \eta}$, and say $\eta$ is an index function of $G$ instead of an index function for $P_G({\vec{x}})$. 

\begin{definition}
	\label{def-nice} Assume $G$ is a plane graph and   $e=v_1v_2$ is a boundary edge of $G$. An index function   $\eta$ of $G-e$ is a  {\em nice  } for $(G,e)$ if the following hold:
	\begin{itemize}
		\item $c_{ {G-e}, \eta} \ne 0$.
		\item  $\eta(v_1) = \eta(v_2)=0$, $\eta(v) \le 2$ for every other boundary vertex $v$, and $\eta(v) \le 4$ for every interior vertex $v$.
	\end{itemize}  
\end{definition}

If $\eta$ is a nice index function for $(G,e)$, then let $\eta'(x)=\eta(x)$ except that $\eta'(x_{v_1})=1$. As $P_G(\vec{x}) = (x_{v_1}-x_{v_2})P_{G-e}(\vec{x})$ and $\eta'(v_2)=0$, we know that $c_{G,\eta'} =c_{G-e, \eta} \ne 0$. Note that $\eta'(x_{v}) < 5$ for each vertex $v$. Thus Theorem \ref{thm-main0} follows from Theorem \ref{thm-2} below.

\begin{theorem}
	\label{thm-2}
	Assume $G$ is a plane graph and   $e=v_1v_2$ is a boundary edge of $G$. Then there exists a nice index function $\eta$ for $(G,e)$. 
\end{theorem}

 A variable $x$ is a {\em dummy variable} in $P({\vec{x}})$ if  $x$ does not really occur in $P({\vec{x}})$, or equivalently, $\eta(x)=0$ for each monomial ${\vec{x}}^{\eta}$ in the expansion of $P$ with a nonzero $c_{P, \eta}$.    We shall frequently need to consider the summation and the product of polynomials.
By introducing dummy variables, we assume   the involved polynomails in the sum or the product have the same set of variables. For example, 
we may view $x_2^2$ be the same as $x_1^0x_2^2x_3^0\ldots x_n^0$, i.e,  $x_2^2={\vec{x}}^{\eta}$, where the variable set is $X=\{x_1, x_2, \ldots, x_n\}$ and $\eta(x_2)=2$, $\eta(x_i)=0$ for $i \ne 2$. We  denote by $X$ the set of variables for   polynomials in concern. For two index functions $\eta_1, \eta_2$,
we write $\eta_1 \le \eta_2$ if $\eta_1(x) \le \eta_2(x)$ for all $x \in X$, and $\eta= \eta_2-\eta_1$ means that $\eta(x) = \eta_2(x)-\eta_1(x)$ for all $x \in X$.

 \begin{observation}
 	\label{obs}
 \mbox{}	  
 	 \begin{enumerate}
 	 	\item If $P(\vec{x})=\alpha P_1(\vec{x})+ \beta P_2(\vec{x})$, then  $c_{P, \eta}=\alpha c_{P_1, \eta} + \beta c_{P_2, \eta}.$ 
 	 	\item If $P({\vec{x}}) =  {\vec{x}}^{\eta'}P_1({\vec{x}})$, then 
 	 	$c_{P, \eta}=  c_{P_1, \eta-\eta'}.$ 
 	 	\item If $P({\vec{x}}) =   {\vec{x}}^{\eta'}P_1({\vec{x}})$ and $\eta' \not\le \eta$, then $c_{P, \eta}=0$.
 	 	\item  If $P({\vec{x}}) =  P_1({\vec{x}})P_2({\vec{x}})$ and for any $\eta'$ with $c_{P_2,\eta'}\ne 0$,  there is a dummy variable $x$ of $P_1({\vec{x}})$ such that $\eta'(x) \ne \eta(x)$,  then $c_{P, \eta}=0$.
 	 	\item If $G$ is a graph and $c_{ G, \eta} \ne 0$, then $\sum_{x \in X} \eta(x) = |E(G)|$. 
 	 \end{enumerate} 
 \end{observation}

\noindent
{\bf Proof of Theorem \ref{thm-2}}
	Assume the theorem is not true and $G$ is a minimum counterexample.
	
	First we consider the case that $G$ has a chord  $e'=xy$. Let $G_1, G_2$ be the two $e'$-components  (i.e., $G_1, G_2$ are induced subgraphs of $G$ with  $V(G) = V(G_1) \cup V(G_2)$ and $V(G_1) \cap V(G_2) = \{x,y\}$) with $e \in G_1$.

	By the minimality of $G$, there exist  a  nice index function $\eta_1$ for $(G_1,e)$, and a nice index function $\eta_2$ for   $(G_2,e')$.
	Let $\eta=\eta_1+\eta_2$. Note that  $$	P_{G-e}(\vec{x})= P_{G_1-e}(\vec{x}) P_{G_2-e'}(\vec{x}).$$  
Let	 $R(\vec{x}) =  c_{ {G_2-e}, \eta_2} {\vec{x}}^{\eta_2}$, $Q(\vec{x}) = P_{G_2-e'}({\vec{x}}) -R(\vec{x})$, 
$P_1({\vec{x}}) =  R(\vec{x}) P_{G_1-e}(\vec{x})$ and $P_2({\vec{x}}) = Q(\vec{x})P_{G_1-e}(\vec{x}) $.   Then $P_{G-e}({\vec{x}}) = P_1({\vec{x}})+P_2({\vec{x}})$.

	By (5) of Observation \ref{obs}, for any index function $\eta'$ with $c_{Q,\eta'} \ne 0$, we have $\eta' \ne \eta_2$ and hence there is a vertex $v \in V(G_2) -\{x,y\}$ such that $\eta'(x_v) \ne \eta_2(x_v) = \eta(x_v)$. As $x_v$ is a dummy variable in $P_{G_1-e}({\vec{x}})$, by (4) of Observation \ref{obs},   we have $c_{P_2,\eta}=0$. By (1) and (2) of Observation \ref{obs}, 
	$$c_{ {G-e}, \eta}  = c_{ P_1, \eta}  = c_{  {G_2-e'}, \eta_2}   c_{ {G_1-e}, \eta_1}\ne 0.$$
	So  $\eta$ is a nice index function for $(G,e)$.
	 
	Assume $G$ has no chord and assume $B(G)=(v_1, v_2, \ldots, v_n)$.

Let $G'=G-v_n$. 
	Let $v_1, u_1, u_2, \ldots, u_k, v_{n-1}$ be the neighbours of $v_n$. Let $$S({\vec{x}}) = (x_{v_n}-x_{v_1})(x_{v_{n-1}}-x_{v_n })( x_{u_1}-x_{v_n }) \ldots ( x_{u_k}-x_{v_n }).$$ Then 
		  $P_{G-e}(\vec{x})=S({\vec{x}})P_{G'-e}(\vec{x})   .$ 
	 
	 If $n=3$, then let $\eta' $ be nice for $(G',e)$. Let 
	 $\eta(x_v)=\eta'(x_v)$ for $v \notin \{u_1,u_2,\ldots, u_k\}$ and $\eta(x_v) = \eta'(x_v)+1$ for 
	 $v \notin \{u_1,u_2,\ldots, u_k\}$ and $\eta(x_{v_n})=2$. 
	 Let $\eta''(x_{v_n})=2$, $\eta''(x_{u_i})=1$ for $i=1,2,\ldots,k$ and $\eta''(x)=0$ for other $x$. 
	 Then $$S({\vec{x}}) = - {\vec{x}}^{\eta''} +x_{v_1}A(\vec{x}) + x_{v_2}B(\vec{x})+x_{v_n}^3C({\vec{x}})$$
	 for some polynimals $A(\vec{x})$,  $B(\vec{x})$ and $C({\vec{x}})$.  Let $P_1({\vec{x}}) = x_{v_1}A(\vec{x})P_{G'-e}({\vec{x}})$,  $P_2({\vec{x}}) = x_{v_2}B(\vec{x})P_{G'-e}({\vec{x}})$ and $P_3({\vec{x}}) = x^3_{v_n}C(\vec{x})P_{G'-e}({\vec{x}})$.
	 As $\eta(x_{v_1})=\eta(x_{v_2})=0$ and $\eta(x_{v_n})=2$, it follows from (3) of Observation \ref{obs} that   $c_{P_1,\eta}=c_{P_2,\eta}=c_{P_3,\eta}=0$. By (1) and (2) of Observation \ref{obs}, 
	  $c_{G-e, \eta} = -c_{G'-e, \eta'} \ne 0$. Hence $\eta$ is nice for $(G,e)$.
	 
	 Assume $n \ge 4$.

		  We say an index function $\eta'$ for $G'-e$   {\em special} if 
		    $\eta'(v_{n-1}) \le 1$, $\eta'(v_1) = \eta'(v_2)=0$, $\eta'(u_j) \le 3$ for $j=1,2,\ldots, k$.
		  	  $\eta'(v) \le 2$ for each other boundary vertex $v $ and $\eta'(v) \le 4$ for each interior vertex  $v$.

	\bigskip 
	\noindent
	{\bf
		Case 1.}  $c_{ {G'-e}, \eta'}  \ne 0$ for some   special index function $\eta'$ for $G'-e$.

  Let $\eta(v) = \eta'(v)$ for $v \notin \{u_1, u_2, \ldots, u_k, v_{n-1}\}$ and $\eta(v) = \eta'(v)+1$ for $v \in \{u_1, u_2, \ldots, u_k, v_{n-1}\}$ and $\eta(v_n)=1$. Let $\eta''$ be the index function defined as $\eta''(x_{v_n}) = \eta''(x_{v_{n-1}}) = \eta''(x_{u_1}) = \ldots = \eta''(x_{u_k})=1$ and $\eta''(x)=0$ for other variables $x$.
 Then $$S({\vec{x}}) =  {\vec{x}}^{\eta''} +x_{v_1}A(\vec{x}) + x_{v_n}^2 B(\vec{x})$$
  for some polynimals $A(\vec{x})$ and $B(\vec{x})$. 
  Let $P({\vec{x}}) = {\vec{x}}^{\eta''}P_{G'-e}({\vec{x}})$, $P_1(\vec{x})=  x_{v_1}A(\vec{x})P_{G'-e}(\vec{x})$ and $P_2(\vec{x})= x^2_{v_n}B(\vec{x})P_{G'-e}(\vec{x}) $.
   Then  
   $$P_{G,e}({\vec{x}}) = P({\vec{x}})+P_1({\vec{x}})+P_2({\vec{x}}).$$ As $\eta(v_1)=0$ and $\eta(v_n)=1$,     it follows from (3) of Observation \ref{obs} that
   $c_{P_1, \eta  } =  c_{P_2, \eta  }=0.$  
 By (1) and  (2) of Observation \ref{obs}, we have 
	 $c_{  {G-e}, \eta}=c_{ P, \eta} =   c_{  {G'-e}, \eta'} \ne 0.$ 
	Hence $\eta$ is nice for $(G,e)$.

	\bigskip 
	\noindent
	{\bf
		Case 2.}  $c_{ {G'-e}, \eta'} = 0$ for every special index function $\eta'$ for $G'-e$.

By the minimality of $G$, there is an index function $\eta''$  nice for $(G',e)$.
  Let  $\eta(x_v)=\eta''(x_v)$ for $v \notin \{u_1,u_2,\ldots, u_k\}$ and $\eta(x_v) = \eta''(x_v)+1$ for 
  $v \notin \{u_1,u_2,\ldots, u_k\}$ and $\eta(x_{v_n})=2$. 
	Let $\eta'''(x_{v_n})=2$, $\eta'''(x_{u_i})=1$ for $i=1,2,\ldots,k$ and $\eta''(x) = 0$ for other $x$. Then 
     $$S({\vec{x}})=  {\vec{x}}^{\eta'''}  +x_{v_1}A(\vec{x}) + x_{v_{n-1}} B(\vec{x})+x_{v_n}^3C({\vec{x}})$$
     for some polynimals $A(\vec{x})$, $B(\vec{x})$ and $C(\vec{x})$.
     
     Let $P(\vec{x}) =  {\vec{x}}^{\eta'''} P_{G'-e}  (\vec{x})$,
      $P_1(\vec{x}) =
      x_{v_1}A(\vec{x}) P_{G'-e}(\vec{x})$, $P_2(\vec{x}) =x_{v_{n-1}}B(\vec{x}) P_{G'-e}(\vec{x}) $ and $P_3(\vec{x}) = x_{v_n}^3C(\vec{x}) P_{G'-e}(\vec{x})$.
As $\eta(x_{v_1})=0$ and $\eta(x_{v_n})=2$,  it follows from (3) of Observation \ref{obs} that $c_{ P_1, \eta  }  =c_{P_3, \eta}=0.$  
As $\eta(v_{n-1})\le 2$, $\eta(u_i) = \eta''(u_i)+1 \le 3$,  it follows from (2) of Observation \ref{obs} that
 $c_{P_2, \eta } = c_{  {G'-e}, \eta'}$ 
for a special index function $\eta'$ for $G'-e$ (note that $\eta'(x_{v_{n-1}}) = \eta(x_{v_{n-1}})-1$ and $\eta'(x) \le \eta(x)$ for other $x$).   By our assumption,  $c_{  {G'-e}, \eta'} = 0$. Therefore 
 $c_{P_2, \eta }  =0 $.  
As
$$P_{G-e}({\vec{x}}) = P({\vec{x}}) + P_1({\vec{x}})+P_2({\vec{x}})+P_3({\vec{x}}),$$ by (1) and (2) of Observation \ref{obs}, 	 $c_{ {G-e}, \eta}=c_{P, \eta} =   c_{ {G'-e}, \eta''} \ne 0.$ 
\qed

\section{An alternate proof}

A digraph   $D$ is {\em Eulerian} if $d_D^+(v)=d_D^-(v)$ for every vertex $v$. Assume $G$ is a graph and $D$ is an orientation of $G$. 
Let  $EE(D)$ (respectively, $OE(D)$) be the set of  spanning Eulerian sub-digraphs of $D$ with an even (respectively, an odd) number of edges.     
 Alon and Tarsi \cite{AT} showed that for an index function $\eta$ of $G$, $c_{G, \eta} = \pm ( |EE(D)|-|OE(D)|)$ for an orientation $D$ with $d_D^+(v) = \eta(v)$ for every $v \in V(G)$. 
 Thus to prove that $c_{G, \eta} \ne 0$ is equivalent to show that there is an orientation $D$ of $G$ with $d_D^+(v) = \eta(v)$ for every $v \in V(G)$ for which $|EE(D)| \ne |OE(D)|$.

  \begin{definition}
  	\label{def-nice} Assume $G$ is a plane graph and   $e=v_1v_2$ is a boundary edge of $G$. An orientation $D$ of $G-e$ is a  {\em nice  } for $(G,e)$ if the following hold:
  	\begin{itemize}
  		\item $|EE(D)| \ne |OE(D)|$.
  		\item  $d_D^+(v_1) = d_D^+(v_2)=0$, $d_D^+(v) \le 2$ for every other boundary vertex $v$, and $d_D^+(v) \le 4$ for every interior vertex $v$.
  	\end{itemize}  
  \end{definition}
  
The following theorem is just a restatement of Theroem \ref{thm-2}, and its  proof  is essentially the same as the proof of Theorem \ref{thm-2}. However, the translation from calculating the coefficients of a polynomial to counting Eulerian subgraphs is not completely trivial. We include a proof of this statement for pedagogical   reason.

\begin{theorem}
	\label{thm-main}
	Assume $G$ is a plane graph and $e=v_1v_2$ is a boundary edge of $G$, then $(G,e)$ has a nice orientation.
\end{theorem}
\begin{proof}
	Assume the theorem is not true and $G$ is a minimum counterexample.
	
	First we consider the case that $G$ has a chord  $e'=xy$. Let $G_1, G_2$ be the two $e'$-components  with $e \in G_1$.
	
	By the minimality of $G$, $(G_1,e)$ has a nice orientation   $D_1$, and $(G_2,e')$ has a nice orientation $D_2$. Let $D=D_1 \cup D_2$. Edges in $D_2$ incident to $x,y$ are not contained in any directed cycles, and hence are not contained in any Eulerian sub-digraph of $D$. 
	Therefore  \begin{eqnarray*}
		EE(D) &=& EE(D_1) \times EE(D_2) + OE(D_1) \times OE(D_2), \\
		OE(D) &=&  EE(D_1) \times OE(D_2) + OE(D_1) \times EE(D_2).
	\end{eqnarray*}
	
	So $|EE(D)|-|OE(D)| = (|EE(D_1)|-|OE(D_1)|)(|EE(D_2)|-|OE(D_2)|) \ne 0$.
	Hence $D$ is a nice orientation of $(G,e)$.
	
	Assume $G$ has no chord. Assume $B(G)=(v_1, v_2, \ldots, v_n)$.
	
	Let $G'=G-v_n$. If $n=3$, i.e. $B(G)$ is a triangle, then let $D'$ be a nice orientation for $(G',e)$, and let $D$ be obtained from $D'$ by adding arcs $(v_3,v_1)$ and $(v_3, v_2)$. As $v_1, v_2$ are sinks, no edge incident to $v_n$ is contained in a directed cycle and hence  $EE(D)=EE(D')$ and $OE(D)=OE(D')$. Therefore $D$ is a nice orientation for $(G,e)$. 
	
	Assume $n \ge 4$.
	Let $v_1, u_1, u_2, \ldots, u_k, v_{n-1}$ be the neighbours of $v_n$.
	
	We call an orientation $D$ of $(G',e)$   {\em special} if the following hold:
	\begin{itemize}
		% \item $|EE(D)| \ne |OE(D)|$.
		\item $v_1, v_2$ has out-degree $0$,
		$v_{n-1}$ has out-degree at most $1$, and each of $u_1,u_2, \ldots, u_k$ has out-degree at most $3$,
		each other boundary vertex has out-degree at most $2$.
		\item  every interior vertex has out-degree at most $4$.
	\end{itemize}
	
	\bigskip 
	\noindent
	{\bf
		Case 1.} $(G',e)$ has a special orientation $D'$ with $|EE(D')| \ne |OE(D')|$.
	
	Let $D$ be the orientation of $G-e$ which is obtained from $D'$ by adding arcs $$(v_n, v_1), (v_{n-1}, v_n), (u_1, v_n), \ldots, (u_k, v_n).$$
	Then $D$ is a nice orientation of $(G,e)$, as $EE(D')=EE(D)$ and $OE(D')=OE(D)$.
	
	\bigskip 
	\noindent
	{\bf
		Case 2.} For any special orientation $D'$ of $(G',e)$, $|EE)D')| = |OE(D')|$.   
	
	By the minimality of $G$, $(G',e)$ has a nice orientation $D''$. Let $D$ be the orientation of $G-e$   obtained from $D''$ by adding arcs $(v_n, v_1), ( v_n, v_{n-1}), (u_1, v_n), \ldots, (u_k, v_n)$.

	For $i=1,2,\ldots, k$, let $$EE_i(D) = \{H \in EE(D): (u_i,v_n) \in H\}, \ OE_i(D)  = \{H \in OE(D): (u_i,v_n) \in H\}.$$

	For $i=1,2,\ldots, k$, if $ EE_i(D) \cup OE_i(D) \ne \emptyset$, then let $C_i$ be a directed cycle in $D$ containing $(u_i,v_n)$.  Note that every directed edge of an Eulerian digraph   is contained in a directed cycle, and $C_i$ must contain $(v_n,v_{n-1})$. 
	Let $D'_i$ be the orientation of $G'$ which is obtained from $D''$ by reversing the direction of edges in $C_i \cap D''$ (note that $C_i \cap D'' =C_i- \{(u_i,v_n), (v_n, v_{n-1})\}$ is a directed path from $v_{n-1}$ to $u_i$).
	
	Observe that $D'_i$ is a special orientation of $(G',e)$. Hence
	  $|EE(D'_i)| = |OE(D'_i)|$.
	
	Now we show that $|EE_i(D)| = |OE_i(D)|$ for $i=1,2,\ldots,k$. If 
	$ EE_i(D) \cup OE_i(D) = \emptyset$, then this is trivially true.
	
	Assume  $ EE_i(D) \cup OE_i(D) \ne \emptyset$.
	
	For each $H \in EE_i(D) \cup OE_i(D)$,   $H \vartriangle C_i^{-1} \in EE(D'_i) \cup OE(D'_i)$.
	Here $C_i^{-1}$ is the reverse of $C_i$ and $H \vartriangle C_i^{-1}$ is the symmetric difference of $H$ and $C_i^{-1}$, i.e., the digraph
	obtained from the edge disjoint union of     $H$ and $C^{-1}_i$   by deleting digons. Note that the symmetric difference of any two Eulerian digraphs is an Eulerian digraph.  Moreover, any edge of $C_i$ contained in $H$ will form a digon with the corresponding edge in $C_i^{-1}$ and hence is deleted. In particular, $(u_i,v_n), (v_n, v_{n-1})$ are edges of $C_i$ contained in $H$ and are deleted. 
	So $H \vartriangle C_i^{-1}$ is a sub-digraph of $D''$.

	Similarly, for each $H \in EE(D'_i) \cup OE(D'_i)$, $H \vartriangle C_i  \in EE_i(D) \cup OE_i(D)$. As $(H \vartriangle C_i ) \vartriangle C_i^{-1} = H$, $\phi(H) := H \vartriangle C_i^{-1}$ is   a one-to-one correspondence between $EE(D'_i) \cup OE(D'_i)$ and $EE_i(D) \cup OE_i(D)$. If   $C_i$ is of even length, then $|E(\phi(H))|$ and $|E(H)|$ have the same parity; if $C_i$ is of odd length, then  $|E(\phi(H))|$ and $|E(H)|$ have different same parities. So if $C_i$ is of even length, then
	$|EE_i(D)| = |EE(D'_i)|, |OE_i(D)| = |OE(D'_i)|$;
	if $C_i$ is of odd length, then $|EE_i(D)| = |OE(D'_i)|, |OE_i(D)| = |EE(D'_i)|$.
	In any case, $|EE_i(D)| = |OE_i(D)|$.
	
	Now $$EE(D) = EE(D'') \cup \bigcup_{i=1}^k EE_i(D), \ \ OE(D) = OE(D'') \cup \bigcup_{i=1}^k OE_i(D).$$
	The unions above are disjoint unions. So $|EE(D)|-|OE(D)| = |EE(D'')|-|OE(D'')| \ne 0$.
\end{proof}

I would like to thank Grzegorz Gutowski for bringing this problem to my attention, and thank Jaroslaw Grytczuk for pointing out that the problem was contained in \cite{HE}.

\bibliographystyle{unsrt}

\begin{thebibliography}{99}
	
	\bibitem{AT}
	N. Alon, M. Tarsi, {\em Colorings and orientations of graphs}, Combinatorica 12 (2) (1992) 125--134. 
	
	\bibitem{HE} D. Hefetz, {\em On two generalizations of the Alon–Tarsi polynomial
	method}, Journal of Combinaotrial Theory Ser. B, 101 (2011) 403 --414.
\bibitem{JT} T. Jensen, B. Toft, {\em Graph Coloring Problems}, Wiley, New York, 1995.

\bibitem{Scha} U. Schauz, {\em Mr. Paint and Mrs. Correct}, Electron. J. Combin. 16 (1) (2009) R77.

\bibitem{Sch} 
U. Schauz, {\em  A paintability version of the Combinatorial Nullstellensatz and list colorings of $k$-partite $k$-uniform hypergraphs},
Electron. J. Combin. 17 (1) (2010) R176.

\bibitem{TH} C. Thomassen, {\rm Every planar graph is 5-choosable}, Journal of Combinatorial Theory
	Ser. B, 62(1) (1994):180--181.
		 
		 \bibitem{Zhu} X. Zhu, {\em  On-line list colouring of graphs}, Electron. J. Combin. 16 (1) (2009) R127.
		 	\end{thebibliography}

\end{document}